\documentclass[12pt]{article}

\usepackage{amssymb}
\usepackage{authblk}
\usepackage{amsmath}
\usepackage{graphicx}
\usepackage{caption}
\usepackage{subcaption}
\usepackage{xcolor}
\usepackage{amsthm}
\usepackage[square,numbers]{natbib}
\usepackage{natbib}
\usepackage{hyperref}
\usepackage[babel]{csquotes}
\usepackage{enumitem}
\usepackage{soul}
\newtheorem{theorem}{Theorem}
\newtheorem{lemma}{Lemma}
\newtheorem{remark}{Remark}
\newtheorem{definition}{Definition}
\newtheorem{proposition}{Proposition}
\newtheorem{example}{Example}
\newtheorem{corollary}{Corollary}

\title{Asymptotics of running maxima for $\varphi$-subgaussian random double arrays}

\author{
Nour Al Hayek \textsuperscript{1}\thanks{ N. Al Hayek.  Email:nouralhayek@gmail.com} \
Illia Donhauzer\textsuperscript{2}\thanks{I. Donhauzer. Email:I.Donhauzer@latrobe.edu.au}   \
Rita Giuliano \textsuperscript{3}\thanks{R. Giuliano.  Email:rita.giuliano@unipi.it} \
Andriy Olenko\textsuperscript{2}\thanks{ A. Olenko. Email:a.olenko@latrobe.edu.au} \
Andrei Volodin \textsuperscript{1}\thanks{A. Volodin.  Email:andrei@uregina.ca}}
\affil{\textsuperscript{1} University of Regina, Regina, Canada

\textsuperscript{2} La Trobe University, Melbourne, Australia

\textsuperscript{3} Universit{\`a} di Pisa, Pisa, Italy}

\date{}
\begin{document}
\maketitle{}
\centerline{\it Dedicated to the memory of Professor Yuri Kozachenko (1940-2020)}

\begin{abstract}
The article studies the running maxima $Y_{m,j}=\max_{1 \le k \le m, 1 \le n \le j} X_{k,n} - a_{m,j}$ where $\{X_{k,n}, k \ge 1, n \ge 1\}$ is a double array of $\varphi$-subgaussian random variables and $\{a_{m,j}, m\ge 1, j\ge 1\}$ is a double array of constants.
Asymptotics of the maxima of the double arrays of positive and negative parts of $\{Y_{m,j}, m \ge 1, j \ge 1\}$ are studied, when $\{X_{k,n}, k \ge 1, n \ge 1\}$ have suitable ``exponential-type'' tail distributions. The main results are specified for various important particular scenarios and classes of $\varphi$-subgaussian random variables.
\end{abstract}



\section{Introduction}
The main focus of the this investigation is to obtain convergence theorems for the running maxima of $\varphi$-subgaussian random variables. The roots of the subject are in classical probability theory, and can be traced back to Gnedenko’s theory of limiting behaviour of maximas of random variables. We refer to the excellent book by Embrechts, Kl\"{u}ppelberg, and Mikosch \cite {EKM} that contains classical and more recent results on limit theorems of maximas of random variables with numerous examples of important practical applications in finance, economics, insurance and other fields.

Let $\{X_{k,n}, k \ge 1, n \ge 1\}$ be a double array of centered random variables (a 2D random field defined on the grid $\mathbb{N}\times\mathbb{N}$) that are not necessarily independent or identically distributed. We assume that these variables are defined on the same probability space $\big\{\Omega, \mathcal{F}, P\big\}.$

Studying properties of normalised maxima of random sequences and processes is one of the classical problems in probability theory that attracted considerable interest in the literature, see, for example, \cite{kra, lead, Pit, Tal} and the references therein. The known asymptotic results broadly belong to three classes that use different probabilistic tools to study properties of 
\begin{itemize}
	\item[(1)] expected maxima (see, for example, \cite{BM} and \cite{Tal}), 
		\item[(2)] convergence almost surely (see \cite{RLA} and \cite{Pi}), 	
		\item[(3)] asymptotic distributions of normalised maximas (see, for example, a comprehensive collection of results in \cite{Pit}).
		\end{itemize}  The case of Gaussian random variables has been extensively investigated for each of these classes.
However, there are still numerous open problems, in particular about an extension of the known results to non-Gaussian scenarios and multidimensional arrays.

This article studies sufficient conditions on the tail distributions of $X_{k,n}$ that guarantee the existence of such a sequence $\{a_{m,j}, m\ge 1, j\ge 1\}$ that the random variables
$$ Y_{m,j}=\max_{1 \le k \le m, 1 \le n \le j} X_{k,n} - a_{m,j}$$ converge to 0 almost surely as the number of random variables $X_{k,n}$ in the above maximum tends to infinity. This type of convergence is called the convergence of running maxima.

Contrary to the majority of classical results on the limiting behaviour of the maxima of random variables, where the convergence in distribution was considered (the third item above), we are interested in the almost surely convergence to zero.  First results of this type were obtained by Pickands \cite{Pi}, where the classical case of Gaussian random variables was considered. Later this result was generalized to wider classes of distributions. In \cite{Ri} running maxima of one-dimensional random sequences were considered and the generalization to the subgaussian case was studied. In \cite{RLA}, the results of \cite{Ri} were generalized to the case of $\varphi$-subgaussian random variables. For recent publications on this subject we refer to Giuliano and Macci \cite{GM}, Cs\'{a}ki and Gonchigdanzan~\cite{CG} and references therein.

The class of subgaussian and $\varphi$-subgaussian random variables is a natural extension of the Gaussian class. The popularity of the Gaussian distribution was justified by the central limit theorem for sums of random variables with small variances. However, asymptotics can be non-Gaussian if summands have large variances. Nevertheless, $\varphi$-subgaussianity still can be an appropriate assumption. Numerous probability distributions belong to the $\varphi$-subgaussian class. For example, reflected Weibull, centered bounded-supported, and sums of independent Gaussian and centered bounded random variables are in this class. \mbox{$\varphi$-subgaussian} random variables were introduced to generalize various properties of subgaussian class considered by Dudley \cite{dud0}, Fernique \cite{fern}, Kahane \cite{kah}, Ledoux and Talagrand \cite{led}. Then, several publications used this class of random variables to construct stochastic processes and fields, see \cite{KO2,KO,KOP}. The monograph \cite{BK} discusses subgaussianity and $\varphi$-subgaussianity in detail and provides numerous important examples and properties.

The main aim of this paper is to investigate the convergence of the running maxima of centered double arrays with more general exponential types of the tail distributions of $X_{k,n}$ than in \cite{RLA}. The integrability conditions on the subgaussian function $\varphi$ will obviously change.

The main results of the paper are Theorems~\ref{P1}--\ref{P4}. In these results the array $\{Y_{m,j}, m\ge 1, j\ge 1\}$ is split into two parts:
$$Y_{m,j}^{+} = \max(Y_{m,j}, 0), \ \ Y_{m,j}^{-} = \max(-Y_{m,j}, 0), \ m\ge 1, j\ge 1,$$
and the convergence of the arrays $\{Y_{m,j}^+, m\ge 1, j\ge 1\},$ $\{Y_{m,j}^-, m\ge 1, j\ge 1\}$ is investigated. The obtained results clearly show how the running maxima behaves depending on the right and left tail distributions of $X_{n,k}$. The dependence of the array $\{a_{m,j}, m\ge 1, j\ge 1\}$ on the function $\psi(\cdot)$ which is the Young-Fenchel transform of $\varphi$, and $\varphi$-subgaussian norm of $X_{k,n}$ is demonstrated. The paper also examines the rate of convergence of the positive parts array $\{Y_{m,j}^+, m\ge 1, j\ge 1\}.$

This paper investigates almost sure and $\lim(\max)$ convergence of random functionals of double arrays. More details about these and other types of convergence and their applications can be found in the publications \cite{DOA, HRVZ, HRV, Kl} and the references therein.

The novelty of the paper compared to the known in the literature results for one-dimensional sequences are:

- the case of random double arrays is studied,

- $\lim(\max)$ convergence is used,

- $\varphi$-subgaussian norms of random variables in the arrays can unboundedly increase,

- conditions on exponential-type bounded tails are weaker than in the literature,

- several assumptions are less restrictive than even in the known results for the one-dimensional case,

- specifications for various important cases and particular scenarious are provided.

This paper is organized as follows. Section~\ref{Prel} provides required definitions and notations. The main results of this article are proved in Section~\ref{mx} and ~\ref{conv}. Conditions on the convergence of running maxima are presented in Section~\ref{mx}. Estimates of the rate of convergence are given in Section ~\ref{conv}. Specifications of the main results and important particular cases are considered in Section \ref{sec5}. Section \ref{sec6} presents some simulation studies. Finally, conclusions and some problems for future investigations are given in the last section.

Throughout the paper, $u \vee v$ denotes $\max(u, v),$ $\mathbb{R}^+$ stands for the set of positive real numbers, and $C$ represents a generic finite positive constant, which is not necessarily same in each appearance.

All computations, plotting, and simulations in this article were performed using the software~R version~4.0.3. A reproducible version of the code in this paper is available in the folder \enquote{Research materials} from the website \url{https://protect-au.mimecast.com/s/w-hDCk8vzZULyROOuVK_Qw?domain=sites.google.com}.

\section{Definitions and auxiliary results}
\label{Prel}
This section presents definitions, notations, and technical results that will be used in the proofs of the main results later.

For double arrays of random variables, due to the lack of linear ordering of $\mathbb{N} \times \mathbb{N}$, there are multiple ways to define different modes of convergence. See the monograph~\cite{Kl} for a comprehensive discussion.

This paper considers $\lim(\max)$ convergence. Let $\{a_{m,j}, m \ge 1, j \ge 1\}$ be a double array of real numbers.

\begin{definition} The array $\{a_{m,j}, m \ge 1, j \ge 1\}$ converges to $a \in \mathbb{R}$ as $m \vee j \to \infty$ if for every $\varepsilon > 0$ there exists an integer $N$ such that if $m \vee j \ge N$ then
\[|a_{m,j}-a| < \varepsilon.\]
\end{definition}

In the following this convergence will be denoted by $\lim_{m \vee j \to \infty} a_{m,j} = a$ or by $a_{m,j} \to a$ as $m \vee j \to + \infty$.

This paper uses the next notations of $\varphi$-subgaussianity.

\begin{definition} A continuous function $\varphi(x), x\in {\mathbb{R}}$, is called an Orlicz $N$-{function} if

a) it is even and convex,

b) $\varphi(0)=0$,

c) $\varphi(x)$  is a monotone increasing function for $x>0$,

d) $\lim\limits_{x\to 0} \frac{\varphi(x)}{x} =0$ and $\lim\limits_{x\to+\infty} \frac{\varphi(x)}{x} =+\infty$.

\end{definition}
In the following the notation $\varphi(x)$ is used for an Orlicz $N$-function.

\begin{example} 
{\rm The function $\varphi(x)=\frac{|x|^r}{r}, \ r> 1,$ is an Orlicz $N$-function.}
\end{example}

\begin{definition} A function $\psi(x), \ x\in {\mathbb{R}}$, given by $\psi(x):=\sup_{y\in {\mathbb{R}}} \left(xy - \varphi(y) \right)$ is called the Young-Fenchel transform of $\varphi(x)$.
\end{definition}
It is well-known that $\psi(\cdot)$ is an Orlicz $N$-function.

\begin{example}
{\rm If $\varphi(x)=\frac{|x|^r}{r}, \ r>1$, then $\psi(x)=\frac{|x|^q}{q},$ where $\frac{1}{r} + \frac{1}{q} =1$.}
\end{example}
Any Orlicz $N$-function $\varphi(x)$ can be represented in the integral form
\[\varphi(x)= \int_0^{|x|} p_{\varphi}(t)\ dt\]
where $p_{\varphi}(t)$, $t\geq0,$ is its density. The density $p_{\varphi}(\cdot)$ is non-decreasing and there exists a generalized inverse $q_\varphi(\cdot)$ defined by
\[q_{\varphi}(t):= \sup\{ u\ge 0: p_{\varphi}(u)\le t\}.\]
Then,
\[\psi(x)=  \int_0^{|x|} q_{\varphi} (t)\ dt.\]
As a consequence, the function $\psi(\cdot)$ is increasing, differentiable, and $\psi'(\cdot)=q_{\varphi}(\cdot)$.

\begin{definition}\label{def4} A random variable $X$ is $\varphi$-subgaussian if $E(X)=0$ and there exists a finite constant $a> 0$ such that  $E\exp\left(t X\right) \le \exp\left(\varphi(at)\right)$ for
all $t\in {\bf \mathbb{R}}$.
The $\varphi$-subgaussian norm $\tau_{\varphi}(X)$ is defined as
$$ \tau_{\varphi}(X):=\inf\{a>0: E\exp\left(t X\right) \le \exp\left(\varphi(at)\right), t\in {\mathbb{R}} \}. $$
\end{definition}

The definition of a $\varphi$-subgaussian random variable is given in terms of expectations, but it is essentially a condition on the tail of the distribution. Namely, the following result holds, see \cite[Lemma 4.3, p.~66]{BK}.

\begin{lemma}
\label{l1}
If $\varphi(\cdot)$ is an Orlicz $N$-function and a random variable $X$ is $\varphi$-subgaus\-sian, then for all $x> 0$ the following inequality holds
\[P\big(X \geq x\big) \le \exp\left(-\psi\left(\frac{x}{\tau_{\varphi}(X)}\right)\right).\]
\end{lemma}

\begin{remark} We refer to the monograph {\rm\cite{BK}} where the notion of $\varphi$-subgaussianity was introduced and discussed in detail. Various examples were also provided in {\rm\cite{BK}}. In the case $\varphi(x)=\frac{x^2}{2}$ the notion of $\varphi$-subgaussianity reduces to the classical  subgaussianity (see, for example,  {\rm\cite[Section {\rm 4.29}]{HJ}}).
\end{remark}
For readers' convenience, we present a brief discussion of recent relevant results in the literature on one-dimensional sequences of $\varphi$-subgaussian random variables.

Consider a zero-mean sequence $\{X_k, k \ge 1\}$ of random variables, and set
\[Y_n=\max_{1\le k \le n} X_k - \sqrt{2\ln n}.\]
If $X_k$ are independent and Gaussian random variables, then $\lim_{n\to \infty} Y_n =0$ a.s., see, for instance, \cite{Pi}.

In \cite{RLA} the following proposition was proved for $Y_n^{+} = \max(Y_n, 0)$.

\begin{proposition}
\label{A} Suppose that there exists $\varepsilon_0 >0$ such that for every $\varepsilon \le \varepsilon_0$, the generalized inverse of the density $p_{\varphi}(\cdot)$ of the Orlicz $N$-function $\varphi(\cdot)$ satisfies the conditions
\[\int_{0}^{\infty}q_{\varphi}(x) \exp\big(-\varepsilon q_{\varphi}(x)\big) \ dx <+\infty \]
and
\[\sup_{k\ge 1} \tau_{\varphi}(X_k) = C \le 1.\]
Then  $\lim_{n\to\infty} Y_n^+ =0$ a.s.
\end{proposition}

It is natural to try to extend Proposition \ref{A} to the multidimensional arrays. This is done in the next section.

Next, the behaviour of $Y_n^- =  \max(-Y_n, 0),\ n\ge 1,$ was also studied in \cite{RLA}, but some additional assumptions on the left tail distribution of $X_n$ were required. Unfortunately, these assumptions cannot be derived from the $\varphi$-subgaussianity assumption (see Remark 2 in \cite{RLA}). In contrast to  Proposition \ref{A}, the independence assumption is also required.

\begin{proposition}
\label{B} Assume that $\{X_k, k\ge 1\}$ is a sequence of zero-mean independent random variables and there exists a number $C>0$ such that, for every $k\ge 1$ and all $x>0$, we have
\[P\big(X_k < x\big) \leq \exp\big(-Ce^{-\psi(x)}\big),\]
where $\psi(\cdot)$ is a positive differentiable function with $q(x) = \psi'(x)$ non-decreasing for $x>0$.
Suppose that there exists an $\varepsilon_0 >0$ such that for every $\varepsilon \le \varepsilon_0$ it holds
\[\int_0^{+\infty}\exp\big(\psi(x)-Ce^{\varepsilon q(x-\varepsilon)}\big) q(x) dx<+\infty.\]
Then $\lim_{n\to\infty} Y_n^- =0$ a.s.
\end{proposition}

In the exponential-type tail condition Proposition \ref{B} uses the same function $\psi(\cdot)$ as in the definition of $\varphi$-subgaussianity of $X_k.$ The next section will extend it to the case of arbitrary functions.

\begin{proposition}
\label{C} Let $\{X_k, k \geq 1\}$ be a sequence of $\varphi$-subgaussian random
variables such that $\sup_{k\ge 1} \tau_{\varphi}(X_k) = c$ and let $\alpha> 2- \frac{1}{c}$. Then
\[\sum_{k=1}^{+ \infty} k^{- \alpha} P\big(Y_k^+ >0\big) < + \infty. \]
\end{proposition}

\begin{remark}
The statement of Proposition {\rm\ref{C}} is obvious for $\alpha >1$. Hence, only the case of $\alpha< 1$ i.e. $c<1,$ is interesting.
\end{remark}

Note that \cite{RLA} also examined the rate of convergence of the sequence $\{Y_n^{+}, n\ge 1\}$. In Proposition \ref{C}, the rate of convergence for  $P(Y_n^+ >0)$ is given, while usually only results for $ P\big(Y_n^+ > \varepsilon\big)$, $\varepsilon> 0,$ were obtained in the existing literature. As $ P\big(Y_n^+ > \varepsilon\big)\le P\big(Y_n^+ > 0\big)$ for any  $\varepsilon> 0$ it also follows from the assumptions of the proposition that $\sum_{k=1}^{+ \infty} k^{- \alpha} P\big(Y_k^+ >\varepsilon\big) < + \infty$.

It was also shown in \cite{RLA} that, Proposition \ref{C} is sharp in some sense. Namely, the following result is true.

\begin{proposition}
\label{D} Let $\{X_k, k \geq 1\}$ be a zero-mean sequence of independent random variables and there exists a strictly increasing differentiable function $\psi:{\mathbb R}^+ \to {\mathbb R}^+$ and $t_0 > 0$ such that, for every $k \geq 1$ and $t > t_0$ it holds
\[P\big(X_k > t\big) \geq \exp\big(- \psi(t)\big).\]
Assume that there exists $\varepsilon_0>0$ such that
\[\limsup _{x \to + \infty} \frac{\max_{x\le \xi\le x+\varepsilon_0} \psi'(\xi)}{\psi(x)}=l <\infty. \]
Then, for every real number $\alpha < 1$ and for every $0 < \varepsilon < (1-\alpha)/l$ it holds
\[\sum_{k=1}^{\infty} k^{- \alpha} P\big(Y_k^+ >\varepsilon\big) = + \infty. \]
\end{proposition}

\section{On asymptotic behaviour of running maxima of $\varphi$-subgaussian double arrays}
\label{mx}
In this section, we establish sufficient conditions on the tail distributions of $X_{k,n}$ that guarantee that positive and negative parts of random variables $Y_{m,j}$ converge to 0 almost surely as $m \vee j \to \infty.$

Let $\varphi(\cdot)$ be an Orlicz $N$-function, $p_{\varphi}(\cdot)$ be its density, the function $\psi(\cdot)$ is the Young-Fenchel transform of $\varphi(\cdot)$, and the function $q_{\varphi}(\cdot)$ be the generalized inverse of the density $p_{\varphi}(\cdot)$.

Let us consider a double array (2D random field defined on the integer grid $\mathbb{N}\times\mathbb{N}$) of zero-mean random variables $\{X_{k,n}, k \ge 1, n \ge 1\}$. The next notations will be used to formulate the main results
\begin{eqnarray*}
Y_{m,j}: &=& \max_{1 \le k \le m, 1 \le n \le j} X_{k,n} -a_{m,j}, \\
Z_{m,j}: &=& X_{m,j}-a_{m,j},
\end{eqnarray*} where $a_{m,j}$ is an increasing function with respect to each of $m$ and $j$ variables, where  $m,j\ge 1.$

Let
\[Y^{+}_{m,j}:=\max(Y_{m,j},0) \ \ {\rm{and}} \ \ Y^{-}_{m,j}:=\max(-Y_{m,j},0).\]

Indices $m$ and $j$ of the random variables $Y_{m,j}$ can be viewed as the parameters defining the rectangular observation window $\{(k, n) : k\leq m,n\leq j, \ k,n \in\mathbb{N}\}$ of the random field $X_{k,n}$ on $\mathbb{N}\times\mathbb{N}$.

The following proofs will use the next extension of Lemma 2 from \cite{RLA} to the case of double arrays.

\begin{lemma}
\label{l2}
For any $\varepsilon> 0$
\[\{\omega \in \Omega : Y^+_{m,j} >\varepsilon  \mbox{ i.o.} \} =\{\omega \in \Omega : Z_{m,j}^+ >\varepsilon \mbox{ i.o.}\},\] where i.o. stands for infinitely often.
\end{lemma}

\begin{proof} It is easy to see that
$$\big\{\omega\in\Omega: Z_{m,j}^{+} > \varepsilon \ i.o. \big\} = \big\{\omega\in\Omega: X_{m,j} > \varepsilon + a_{m,j} \ i.o. \big\}$$
$$\subset \big\{\omega \in \Omega: \max_{1 \le k \le m, 1 \le n \le j}X_{k,n} > \varepsilon + a_{m,j} \ i.o. \big\} = \big\{\omega\in\Omega: Y^{+}_{m,j} > \varepsilon \ i.o. \big\}.$$

Also, as $a_{m,j}$ is an increasing function of $m$ and $j$, it holds
$$\big\{\omega\in\Omega: Y_{m,j}^+ > \varepsilon \ i.o. \big\} = \big\{\omega\in\Omega: X_{k,n} > \varepsilon +a_{m,j}, \ {\rm{for}} \ 1\leq k\leq m, 1\leq n\leq j \ i.o. \big\}$$
$$\subset\big\{\omega\in\Omega: X_{k,n} > \varepsilon + a_{k,n} \ i.o. \big\} = \big\{\omega\in\Omega: Z_{m,j}^+ > \varepsilon \ i.o.\big\}, $$ which completes the proof. \end{proof}

\begin{remark}
\label{io}
Let $\{A_{n}\}_{n=1}^\infty$ be an infinite sequence of events. By $\{A_n \ i.o.\}$ we denote an event that  infinitely many events from $\{A_{n}\}_{n=1}^\infty$ holds true. The importance of the notion i.o. can be explained by the following well-known statement, which is crucial for proving the almost sure convergence: $X_{k,n} \to 0$ almost surely, when $k\vee n\to+\infty,$ if and only if for all $\varepsilon > 0, P(|X_{k,n}| \geq \varepsilon \mbox{ i.o.}) = 0.$
\end{remark}
The following result extends Proposition \ref{A} to the case of double arrays of random variables.

\begin{theorem}
\label{P1}
Let $\{X_{k,n}, k \ge 1, n \ge 1\}$ be a double array of $\varphi$-subgaussian random variables and $g(\cdot)$ be a non-decreasing function such that for all $k,n\geq1$
\begin{equation}
\label{supc}
\tau_\varphi(X_{k,n})\leq g(\ln(kn))
\end{equation}
and
\[a_{m,j} = g(\ln(mj))\psi^{-1}(\ln(mj)).\]
Suppose that there exists an $\varepsilon_0 >0$ such that for every $\varepsilon \in (0,\varepsilon_0]$
\begin{equation}
\label{eq1}
\int_{0}^{\infty} \psi(x) q_{\varphi}(x) \exp\left(-\frac{\varepsilon q_{\varphi}(x)}{g(\psi(x)+\ln(2))}\right) dx < +\infty.
\end{equation}

Then $\lim_{m \vee j \to \infty} Y^{+}_{m,j}=0$  a.s.
\end{theorem}

\begin{remark}
In the following, without loss of generality, we consider only  non-degenerated random variables with non-zero $\varphi$-subgaussian norms. Therefore, it holds $g(\cdot)>0.$ For the case of identically distributed $X_{k,n}$ we assume that $g(x) \equiv \tau_\varphi(X_{k,n}) \equiv C>0.$
\end{remark}

\begin{remark}
If the function $g(\cdot)$ is bounded by a constant $C$ from above then \eqref{supc} is the same as the corresponding assumption in Propositions {\rm\ref{A}} and {\rm\ref{C}}.
\end{remark}

\begin{proof}It follows from Lemma \ref{l2}, Remark \ref{io} and the Borel-Cantelli lemma that it is enough to show that for any $\varepsilon>0$
\[\sum_{m=1}^{\infty}  \sum_{j=1}^{\infty} P\big( Z^{+}_{m,j} \geq \varepsilon\big) < \infty, \]
because then $P\big(Z^{+}_{m,j}\geq\varepsilon \mbox{ i.o.} \big)=0$.

Note, that  by Lemma \ref{l1} and assumption \eqref{supc} for all $m,j \in \mathbb{N},$ except a finite number, it holds
$$P \big(Z^{+}_{m,j} \geq \varepsilon \big) = P\big( Z_{m,j} \geq \varepsilon\big) = P\big( X_{m,j} \geq g(\ln(m  j))\psi^{-1}(\ln (m  j))+\varepsilon\big) $$
$$= P\left( \frac{X_{m,j}}{g(\ln(m j))} \geq \psi^{-1}(\ln (m  j))+\frac{\varepsilon}{g(\ln(m j))}\right) $$
$$\leq P\left( \frac{X_{m,j}}{\tau_\varphi(X_{m,j})} \geq \psi^{-1}(\ln (m j))+\frac{\varepsilon}{g(\ln(m j))}\right) $$
$$\le \exp \left(-\psi\left( \psi^{-1}(\ln (m j))+\frac{\varepsilon}{g(\ln(m j))}\right) \right)$$
since $\psi(\cdot)$ is increasing.

Therefore, it is enough to prove that the double sum
\begin{equation}
\label{series}
S:=\sum_{m=1}^{\infty} \sum_{j=1}^{\infty} \exp \left(-\psi\left( \psi^{-1}(\ln (m j))+\frac{\varepsilon}{g(\ln(m j))}\right) \right)
\end{equation}
converges.

Let us fix $m \ge 1$ and investigate the behaviour of the inner sum
\[\sum_{j=1}^{\infty} \exp \left(-\psi\left( \psi^{-1}(\ln (m  j))+\frac{\varepsilon}{g(\ln(m j))}\right) \right).\]
Note that
\[\sum_{j=1}^{\infty} \exp \bigg(-\psi\left( \psi^{-1}(\ln (m  j))+\frac{\varepsilon}{g(\ln(m j))}\right) \bigg)\]
\[=\exp \bigg(-\psi\left( \psi^{-1}(\ln (m))+\frac{\varepsilon}{g(\ln(m))}\right) \bigg) + \sum_{j=2}^{\infty} \exp \bigg(-\psi\left( \psi^{-1}(\ln (m  j))+\frac{\varepsilon}{g(\ln(m j))}\right) \bigg).\]
Now, for $x\in [j-1,j]:$
\[\ln (m x)\le \ln (m j) \quad \mbox{and} \quad \ln(m(x + 1))\ge \ln (m j).\]
Because  $\psi(\cdot),$  $\psi^{-1}(\cdot)$ and $g(\cdot)$ are increasing functions,
\[\psi^{-1}(\ln (m  j))+\frac{\varepsilon}{g(\ln(m j))} \ge \psi^{-1}(\ln (m x)) + \frac{\varepsilon}{g(\ln(m(x + 1)))},\]
\[-\psi\left( \psi^{-1}(\ln (m  j))+\frac{\varepsilon}{g(\ln(m j))}\right) \le -\psi\left( \psi^{-1}(\ln (m x)) + \frac{\varepsilon}{g(\ln(m(x + 1)))}\right),\]
which results in
\[\sum_{j=2}^{\infty} \exp \left(-\psi\left( \psi^{-1}(\ln (m  j))+\frac{\varepsilon}{g(\ln(m j))}\right) \right)\]
\[ \le \int_1^{\infty} \exp \left(-\psi\left( \psi^{-1}(\ln (m x)) + \frac{\varepsilon}{g(\ln(m(x + 1)))}\right) \right)  dx. \]

Therefore, for fixed $m \ge 1$,
\begin{eqnarray}
\label{eq2}
&& \sum_{j=1}^{\infty} \exp \left(-\psi\left( \psi^{-1}(\ln (m  j))+\frac{\varepsilon}{g(\ln(m j))}\right) \right) \nonumber \\
&\le& \exp \left(-\psi\left( \psi^{-1}(\ln (m))+\frac{\varepsilon}{g(\ln(m))}\right) \right) \nonumber \\
&& + \int_1^{\infty} \exp \left(-\psi\left( \psi^{-1}(\ln (m x))+\frac{\varepsilon}{g(\ln(m(x + 1)))}\right) \right)  dx \nonumber \\
&\leq& \exp \left(-\psi\left( \psi^{-1}(\ln (m))+\frac{\varepsilon}{g(\ln(m))}\right) \right) \nonumber \\
&& + \int_1^{\infty} \exp \left(-\psi\left( \psi^{-1}(\ln (m x))+\frac{\varepsilon}{g(\ln(mx)+\ln(2))}\right) \right)dx
\end{eqnarray}
as $\ln(m(x+1))\leq \ln(mx) + \ln(2)$ for $m\geq 1$ and $x \geq 1$.

To study the last integral in \eqref{eq2} we use the substitution $t=\psi^{-1}(\ln (mx))$. Then $x= \frac{\exp(\psi(t))}{m}$ and

$$ \int_{1}^{\infty} \exp \left(-\psi\left(\psi^{-1}(\ln (mx))+\frac{\varepsilon}{g(\ln(mx)+\ln(2))}\right)\right)dx $$
$$\leq \frac{1}{m} \int_{\psi^{-1}(\ln m)}^{\infty} \psi'(t) \exp\left(\psi(t) - \psi\left(t+\frac{\varepsilon}{g(\psi(t)+\ln(2))}\right)\right)dt.$$
By the mean value theorem and $\psi^{\prime}(\cdot)= q_{\varphi}(\cdot)$ it follows that there exists such $\xi \in \left[t,t+\frac{\varepsilon}{g(\psi(t)+\ln(2))}\right]$ that it holds
$$\psi(t)-\psi\left(t+\frac{\varepsilon}{g(\psi(t)+\ln(2))}\right) = -\frac{\varepsilon}{g(\psi(t)+\ln(2))}\psi^{\prime}(\xi)$$
$$= -\frac{\varepsilon}{g(\psi(t)+\ln(2))}q_{\varphi}(\xi) \leq -\frac{\varepsilon q_{\varphi}(t)}{g(\psi(t)+\ln(2))},$$ as $q_{\varphi}(t)$ is a non-decreasing function.

Thus, we obtain the next upper bound
$$ \int_{1}^{\infty} \exp \left(-\psi\left(\psi^{-1}(\ln (mx))+\frac{\varepsilon}{g(\ln(mx)+\ln(2))}\right)\right)dx$$
\begin{equation}
\label{eq3}
\leq \frac{1}{m} \int_{\psi^{-1}(\ln m)}^{\infty} q_{\varphi}(t) \exp\left(- \frac{\varepsilon q_{\varphi}(t)}{g(\psi(t)+\ln(2))}\right) dt.
\end{equation}

Therefore, by \eqref{eq2}, \eqref{eq3} the double sum in \eqref{series} can be estimated as
$$ S \leq \sum_{m=1}^{\infty} \exp\left(-\psi\left(\psi^{-1}(\ln(m)\big)+\frac{\varepsilon}{g(\ln(m))}\right)\right)$$
\begin{equation}
\label{eqlong}
+\sum_{m=1}^\infty\frac{1}{m} \int_{\psi^{-1}(\ln m)}^{\infty} q_{\varphi}(t) \exp\left(-\frac{\varepsilon q_{\varphi}(t)}{g(\psi(t)+\ln(2))}\right) dt.
\end{equation}

Similar to the above computations the first sum in \eqref{eqlong} can be bounded as
$$\sum_{m=1}^{\infty} \exp\left(-\psi\left(\psi^{-1}\big(\ln(m)\big)+\frac{\varepsilon}{g(\ln(m))}\right)\right) \leq \exp\left(-\psi\left(\psi^{-1}(0)+\frac{\varepsilon}{g(0)}\right)\right)$$
$$+\int_{1}^{\infty}\exp\left(-\psi\left(\psi^{-1}(\ln(x)) + \frac{\varepsilon}{g(\ln(x+1))}\right)\right)dx $$
\begin{equation}\label{temp}
 = \exp\left(-\psi\left(\frac{\varepsilon}{g(0)}\right)\right) + \int_0^\infty q_{\varphi}(t) \exp\left(-\frac{\varepsilon q_{\varphi}(t)}{g(\psi(t)+\ln(2))}\right)  dt.
\end{equation}

As $\psi^{-1}(\cdot)$ is an increasing function, for the second sum in \eqref{eqlong} one gets
$$\sum_{m=1}^\infty\frac{1}{m} \int_{\psi^{-1}(\ln(m))}^{\infty} q_{\varphi}(t) \exp\left(-\frac{\varepsilon q_{\varphi}(t)}{g(\psi(t)+\ln(2))}\right) dt$$
$$\leq \int_{0}^{\infty} q_{\varphi}(t) \exp\left(-\frac{\varepsilon q_{\varphi}(t)}{g(\psi(t)+\ln(2))}\right)dt$$
\begin{equation}
\label{eqlong2}
 + \int_0^\infty\frac{1}{u} \int_{\psi^{-1}(\ln (u))}^{\infty} q_{\varphi}(t) \exp\left(-\frac{\varepsilon q_{\varphi}(t)}{g(\psi(t)+\ln(2))}\right) dtdu.
\end{equation}

By substitution $y = \ln(u)$ and changing the order of integration
$$ \int_0^\infty\frac{1}{u} \int_{\psi^{-1}(\ln (u))}^{\infty} q_{\varphi}(t) \exp\left(-\frac{\varepsilon q_{\varphi}(t)}{g(\psi(t)+\ln(2))}\right) dtdu $$
$$= \int_0^{+\infty}\int_{\psi^{-1}(y)}^{+\infty}q_{\varphi}(t)\exp\left(-\frac{\varepsilon q_{\varphi}(t)}{g(\psi(t)+\ln(2))}\right)dt dy$$
\begin{equation}
\label{eqlong3}
=\int_{0}^\infty \exp\left(-\frac{\varepsilon q_{\varphi}(t)}{g(\psi(t)+\ln(2))}\right)\int_0^{\psi(t)}dydt $$
$$= \int_{0}^\infty \exp\left(-\frac{\varepsilon q_{\varphi}(t)}{g(\psi(t)+\ln(2))}\right)\psi(t)dt < +\infty,
\end{equation} where the finiteness of the last integral follows from $\lim_{x\to\infty}\frac{\psi(x)}{x} = + \infty$ and the assumption  \eqref{eq1}.

Combining  \eqref{eqlong3} with \eqref{eqlong}, \eqref{temp}, \eqref{eqlong2} we obtain the convergence of $S$ which completes the proof. \end{proof}

For the case of double array of random variables with bounded $\varphi$-subgaussian norms the function $g(\cdot)$ can be selected identically equal to a constant. Therefore, Theorem \ref{P1} can be specified as follows.

\begin{corollary}
\label{rc}
Let $\{X_{k,n}, \ k\geq 1,\ n\geq 1  \}$ be a double array of $\varphi$-subgaussian random variables with $\sup_{k,n\in\mathbb{N}}\tau_{\varphi}(X_{k,n})\leq 1.$ Suppose, that there exist $\varepsilon_0 > 0$ such that for every $\varepsilon \in (0,\varepsilon_0]$
$$\int_0^{+\infty}\psi(x)q_{\varphi}(x)\exp\left(-\varepsilon q_{\varphi}(x)\right)dx<+\infty.$$ Then,
$$\lim_{m\vee j\to+\infty}\left( \max_{1\leq k \leq m, 1\leq n \leq j} X_{k,n} - \psi^{-1}(\ln(mj))\right)^+ = 0, \ a.s.$$
\end{corollary}

The asymptotic behaviour of the sequence $\{Y_{m,j}^-, \ m,j\ge 1\}$ cannot be described in terms of subgaussianity only. Roughly speaking, an opposite type of the inequality is required (see Remark 2 in \cite{RLA}). Moreover, in addition to the conditions of Proposition \ref{P1}, it is assumed that random variables in the double array are independent. The following result is an extension of Proposition \ref{B} to the case of double arrays.

\begin{theorem}
\label{P2}
Let $\{X_{k,n}, k,n \ge 1\}$ be a double array of independent $\varphi$-subgaussian random variables and the array $\{a_{m,j},m,j\ge 1\}$ and function  $g(\cdot)$ are defined in Theorem~{\rm \ref{P1}}.  Let $\kappa(x)$ be a positive increasing differentiable function with the derivative $r(x) = \kappa'(x)$ non-decreasing for $x>0$.   Assume that there exists $C>0$ such that for every $k,n \ge 1$  and all $x>0$
\[P\left(\frac{X_{k,n}}{g(\ln(kn))} < x\right) \le \exp\big( -Ce^{-\kappa(x)}\big),\] and
\begin{equation}
\label{th2cond}
\psi(x) - \kappa\left(\frac{xg(x)}{g(0)}\right)\geq C_0(x)
\end{equation} for some function $C_0(\cdot).$
Suppose that there exists $A, \ \varepsilon_0 >0$ such that for every $\varepsilon \in (0,\varepsilon_0]$
\[
\int_A^{+\infty}\exp\left( -\frac{Cy}{2}\exp\left( -\kappa\left(\frac{g(\ln(y))}{g(0)}\psi^{-1}(\ln(y)) - \frac{\varepsilon}{g(\ln(y))} \right) \right)\right) < + \infty
\] and
\[\int_{A}^{+\infty}\psi(y)q_{\varphi}(y)\exp\left(\psi(y) - \frac{C}{2}\exp\left( C_0(y) + \frac{ r\left(\frac{yg(\psi(y))}{g(0)}-\frac{\varepsilon}{g(\psi(y))}\right)}{g(\psi(y))} \right) \right)dy < + \infty.
\]

Then $\lim_{m \vee j \rightarrow \infty} Y^{-}_{m,j}=0$  a.s.
\end{theorem}

\begin{proof}
Using the Borel-Cantelli lemma, we will prove that for every $\varepsilon \in (0,\varepsilon_0]$ $P\big(Y_{m,j}^{-} >\varepsilon \mbox{ i. o.}\big)=0$. By the independence of $X_{k,n}$ one gets
$$P\big(Y^{-}_{m,j} > \varepsilon\big) = P\big( Y_{m,j}< -\varepsilon\big)=P\big(\max_{1\le k\le m, 1\le n\le j} X_{k,n}<g(\ln(mj))\psi^{-1}(\ln(mj))-\varepsilon\big) $$
$$ = \prod_{k=1}^m \prod_{n=1}^j P\big(X_{k,n}<g(\ln(mj))\psi^{-1}(\ln(mj))-\varepsilon\big) $$
$$ = \prod_{k=1}^m \prod_{n=1}^j P\left(\frac{X_{k,n}}{g(\ln(kn))}<\frac{g(\ln(mj))}{g(\ln(kn))}\psi^{-1}(\ln(mj))-\frac{\varepsilon}{g(\ln(kn))}\right) $$
$$ \leq \prod_{k=1}^m \prod_{n=1}^j P\left(\frac{X_{k,n}}{g(\ln(kn))}< \frac{g(\ln(mj))}{g(0)}\psi^{-1}(\ln(mj))-\frac{\varepsilon}{g(\ln(mj))}\right) $$
$$ \leq \exp\left( -C m j\exp\left(-\kappa\left(\frac{g(\ln(mj))}{g(0)}\psi^{-1}(\ln(mj))-\frac{\varepsilon}{g(\ln(mj))}\right)\right)\right),$$ where we used the monotonicity of the function $g(\cdot)$ and $g(0)\geq 1.$
Therefore,
$$\sum_{m=1}^{\infty} \sum_{j=1}^{\infty} P \left( Y^{-}_{m,j} > \varepsilon \right) \le \sum_{m=1}^{\infty} \sum_{j=1}^{\infty} \exp\left(-C m j \right.$$
\begin{equation}
\label{est}
\left.\times\exp\left(-\kappa\left(\frac{g(\ln(mj))}{g(0)}\psi^{-1}(\ln(mj))-\frac{\varepsilon}{g(\ln(mj))}\right)\right)\right).
\end{equation}

As the functions $ g(\cdot)$ and $\psi^{-1}(\cdot)$ are non-decreasing, then for any fixed $m \ge 1$  we can majorize the second sum as
\[\sum_{j=1}^{\infty} \exp\left( -C m j\exp\left(-\kappa\left(\frac{g(\ln(mj))}{g(0)}\psi^{-1}(\ln(mj))-\frac{\varepsilon}{g(\ln(mj))}\right)\right)\right) \]
$$\le \exp\left( -C m\exp\left(-\kappa\left(\frac{g(\ln(m))}{g(0)}\psi^{-1}(\ln(m))-\frac{\varepsilon}{g(\ln(m))}\right)\right)\right)$$
$$+ \int_{2}^{\infty} \exp\left( -C m(x-1)\exp\left(-\kappa\left(\frac{g(\ln(mx))}{g(0)}\psi^{-1}(\ln(mx))-\frac{\varepsilon}{g(\ln(mx))}\right)\right)\right) dx.$$
As $x/2\leq x-1$ for $x\geq 2,$ the above integral can be estimated by
$$\int_{2}^{\infty} \exp\left( -\frac{Cmx}{2} \exp\left(-\kappa\left(\frac{g(\ln(mx))}{g(0)}\psi^{-1}(\ln(mx))-\frac{\varepsilon}{g(\ln(mx))}\right)\right)\right) dx.$$

By the change of variables $y =\psi^{-1}(\ln (mx)), \ x = \frac{1}{m}\exp(\psi(y))$, this integral equals
$$\frac{1}{m} \int_{\psi^{-1}(\ln(2m))}^{\infty}\exp\left(\psi(y)-\frac{C}{2}\exp\left( \psi(y) -\kappa\left(\frac{yg(\psi(y))}{g(0)} - \frac{\varepsilon}{g(\psi(y))}\right)\right)\right)q_{\varphi}(y)dy
$$
$$ = \frac{1}{m} \int_{\psi^{-1}(\ln(2m))}^{\infty} \exp \left(\psi(y)-\frac{C}{2}\exp\left( \psi(y) \right. \right. -\kappa\left(\frac{yg(\psi(y))}{g(0)}\right)$$
\begin{equation} 
\label{th2temp}
\left. \left. +\kappa\left(\frac{yg(\psi(y))}{g(0)}\right) - \kappa\left(\frac{yg(\psi(y))}{g(0)} - \frac{\varepsilon}{g(\psi(y))}\right)\right)\right)q_{\varphi}(y)dy.
\end{equation} 

By the mean value theorem, as $r(\cdot)$ is non-decreasing, it holds
$$\kappa\left(\frac{yg(\psi(y))}{g(0)}\right) - \kappa\left(\frac{yg(\psi(y))}{g(0)} - \frac{\varepsilon}{g(\psi(y))}\right) \geq \frac{\varepsilon}{g(\psi(y))}r\left(\frac{yg(\psi(y))}{g(0)} - \frac{\varepsilon}{g(\psi(y))}\right).$$

Thus, applying the above inequality and assumption \eqref{th2cond} one gets the following upper bound for the integral in \eqref{th2temp}
$$\frac{1}{m}\int_{\psi^{-1}(\ln(2m))}^{+\infty}\exp\left(\psi(y)-\frac{C}{2}\exp\left( C_0(y) + \frac{\varepsilon r\left(\frac{yg(\psi(y))}{g(0)} - \frac{\varepsilon}{g(\psi(y))}\right)}{g(\psi(y))}\right)\right)q_{\varphi}(y)dy.$$

Hence, the right hand side of \eqref{est} can be estimated by
$$ \sum_{m=1}^{\infty}  \exp\left( -C m\exp\left(-\kappa\left(\frac{g(\ln(m))}{g(0)}\psi^{-1}(\ln(m))-\frac{\varepsilon}{g(\ln(m))}\right)\right)\right) $$
\begin{small}
$$+ {\sum_{m=1}^{\infty}\frac{1}{m}\int_{\psi^{-1}(\ln(2m))}^{+\infty}\exp\left(\psi(y)-\frac{C}{2}\exp\left( C_0(y) + \frac{\varepsilon r\left(\frac{yg(\psi(y))}{g(0)} - \frac{\varepsilon}{g(\psi(y))}\right)}{g(\psi(y))}\right)\right)q_{\varphi}(y)dy} $$
\end{small}
$$ \leq \exp\left( -C\exp\left(-\kappa\left(\psi^{-1}(0)-\frac{\varepsilon}{g(0)}\right)\right)\right) $$
$$ + \int_{2}^{\infty} \exp\left( -\frac{Cu}{2}\exp\left(-\kappa\left(\frac{g(\ln(u))}{g(0)}\psi^{-1}(\ln(u))-\frac{\varepsilon}{g(\ln(u))}\right)\right)\right)du $$
$$ + \int_{\psi^{-1}(\ln(2))}^{+\infty}\exp\left(\psi(y)-\frac{C}{2}\exp\left(C_0(y) + \frac{\varepsilon r\left(\frac{yg(\psi(y))}{g(0)} - \frac{\varepsilon}{g(\psi(y))}\right)}{g(\psi(y))}\right)\right)q_{\varphi}(y)dy$$

\begin{footnotesize}
\begin{equation}
\label{th2temp2}
+\int\limits_1^{+\infty}\frac{1}{u}\int\limits_{\psi^{-1}(\ln(2u))}^{+\infty}\exp\left(\psi(y)-\frac{C}{2}\exp\left( C_0(y) + \frac{\varepsilon r\left(\frac{yg(\psi(y))}{g(0)} - \frac{\varepsilon}{g(\psi(y))}\right)}{g(\psi(y))}\right)\right)q_{\varphi}(y)dydu.
\end{equation}
\end{footnotesize}

By the change of variables $t = \ln(2u)$ and the change of the order of integration we obtain that the last integral equals
\begin{footnotesize}
\begin{equation}
\label{th2temp3}
\int_{\psi^{-1}(\ln(2))}^{+\infty}(\psi(y)-\ln(2))q_{\varphi}(y)\exp\left(\psi(y)-\frac{C}{2}\exp\left(C_0(y) + \frac{\varepsilon r\left(\frac{yg(\psi(y))}{g(0)} - \frac{\varepsilon}{g(\psi(y))}\right)}{g(\psi(y))} \right) \right)dy
\end{equation}
\end{footnotesize}

Then, the boundedness $\sum_{m=1}^\infty\sum_{j=1}^\infty P(Y_{m,j}^- > \varepsilon) < +\infty$ follows from \eqref{th2temp2}, \eqref{th2temp3} and the assumptions of the theorem.  \end{proof}

For the case of double arrays of random variables with uniformly bounded \mbox{$\varphi$-subgaussian} norms the next specification, holds true.

\begin{corollary}
\label{cor2}
Let $\{X_{k,n}, \ k\geq 1, \ n\geq 1  \}$ be a double array of $\varphi$-subgaussian random variables with $\sup_{k,n\in\mathbb{N}}\tau_{\varphi}(X_{k,n})\leq 1.$ Let $\kappa(x)$ be a positive increasing differentiable function with the derivative $r(x) = \kappa'(x)$ that is non-decreasing for $x>0$ and $\psi(x)-\kappa(x)\geq C_0(x)$ for some function $C_0(\cdot).$ Assume that there exists $C>0$ such that for every $k,n\geq 1$ and $x>0$
$$P(X_{k,n}<x) \leq \exp\big(-C\exp(-\kappa(x))\big).$$
Suppose that there exist constants $A, \varepsilon_0 > 0$ such that for every $\varepsilon \in (0,\varepsilon_0]$
\begin{equation}
\label{cor21}
\int_A^{+\infty}\exp\left(-\frac{Cy}{2}\exp\big( -\kappa(\psi^{-1}(\ln(y))-\varepsilon)  \big) \right)dy < + \infty
\end{equation} and
\begin{equation}
\label{cor22}\int_A^{+\infty}\psi(y)q_{\varphi}(y)\exp\left(\psi(y) -\frac{C}{2}\exp\big(C_0(y)+\varepsilon r(y-\varepsilon) \big) \right)dy < +\infty.\end{equation} Then,
$$\lim_{m\vee j\to +\infty}\left(\max_{1\leq k \leq m, 1 \leq n \leq j}X_{k,n}-\psi^{-1}(\ln(mj))\right)^- = 0 \ \ a.s.$$
 \end{corollary}

\begin{remark}
\label{rm6}
Lemma {\rm\ref{l1}} provides the upper bound on the tail probability of the $\varphi$-subgaussian random variable $X_{k,n}$
$$P\big(X_{k,n} \geq x \big) \leq \exp\big(-\psi(Ax)\big), \ x>0.$$
The condition $P\big( X_{k,n} < x\big) \leq \exp\big(-C\exp(-\kappa(x)) \big)$ in some sence is opposite. Namely, the lower bound on the tail probability
$$P\big( X_{k,n} \geq x \big) \geq C\exp\big(-\kappa(x)\big)$$ implies
$$P\big( X_{k,n} < x \big) \leq \exp\big(-P(X_{k,n}\geq x)\big)\leq \exp\big(-C\exp(-\kappa(x)) \big)$$ as $t\leq \exp(-(1-t)).$
\end{remark}

\begin{theorem}\label{th1}
Assume that $\{X_{k,n}, k \ge 1,  n \ge 1\}$ is a double array of independent $\varphi$-subgaussian random variables. If the assumptions of Theorems~{\rm\ref{P1}} and~ {\rm\ref{P2}} are satisfied, then
$\lim_{m \vee j \to \infty} Y_{m,j}=0$  a.s.
\end{theorem}
The proof of  Theorem \ref{th1} follows from the proofs of Theorems~\ref{P1} and~\ref{P2}.

\section{On convergence rate of running maxima of random double arrays}
\label{conv}

This section investigates the series $$\sum_{m=1}^{+ \infty}\sum_{j=1}^{+ \infty}( mj)^{- \alpha} P(Y_{m,j}^+ > \varepsilon), \ \varepsilon >0.$$ It proves that the series converges for a suitable constant $\alpha.$

The following theorem and corollary are generalizations of Proposition \ref{C} to the case of $\varphi$-subgaussian arrays with not necessary uniformly bounded $\varphi$-subgaussian norms.

\begin{theorem}
\label{P3}
Let $\{X_{k,n},k \geq 1, n \geq 1\}$ be a double array of $\varphi$-subgaussian random
variables such that for all $m,j\geq 1,$ $1\leq k\leq m,\ 1\leq n \leq j,$ and some positive-valued function $f(\cdot)$ it holds
\[\frac{g(\ln(mj))}{\tau_{\varphi}(X_{k,n})} \geq f\left(\frac{mj}{kn}\right) \ge 1\] and
$$\sum_{m=1}^\infty \sum_{j=1}^\infty \sum_{k=1}^m \sum_{n=1}^j \left(mj\right)^{-\alpha-f(\frac{mj}{kn})}< + \infty.$$ Then
\begin{equation}
\label{thp3}
\sum_{m=1}^{\infty}\sum_{j=1}^{\infty}(mj)^{- \alpha} P(Y_{m,j}^+ >0) < + \infty.
\end{equation}
\end{theorem}

\begin{proof}  By Lemma \ref{l1} it follows that
$$P(Y_{m,j}^+ >0) = P\left(\max _{1 \leq k \leq m,1 \leq n \leq j}X_{k,n} > g(\ln(mj))\psi^{-1}\left(\ln (mj)\right)\right)$$
$$\leq \sum_{k=1}^m \sum_{n=1}^j P\left(\frac{X_{k,n}}{\tau_{\varphi}(X_{k,n})} > \frac{g(\ln(mj))\psi^{-1}(\ln(mj))}{\tau_\varphi(X_{k,n})} \right)$$
$$\leq\sum_{k=1}^m \sum_{n=1}^j  \exp \left(-\psi\left(f\left(\frac{mj}{kn}\right)\psi^{-1}(\ln(mj))\right)
\right)$$
$$\leq\sum_{k =1}^m \sum_{n =1}^j  \exp \left(-f\left(\frac{mj}{kn}\right)\ln(mj)\right).
$$
The last inequality follows from $\psi(\theta x) \ge \theta \psi(x), \ \theta\ge 1,$ that is true for any Orlicz $N$-function.

Hence,
$$\sum_{m=1}^{+ \infty}\sum_{j=1}^{+ \infty}(mj)^{- \alpha} P(Y_{m,j}^+ >0) \le \sum_{m=1}^{+ \infty}\sum_{j=1}^{+ \infty} \sum_{k=1}^{m}\sum_{n=1}^{j} (mj) ^{- \alpha -f(\frac{mj}{kn})}<+\infty,$$
by the assumption of the Theorem.
\end{proof}

\begin{corollary}
\label{cor3}
Let the conditions of Theorem {\rm\ref{P3}} are satisfied and $f(x)\geq c_0>0$ for $x\geq 1.$ Then, {\eqref{thp3}} holds true for $\alpha > 2 - c_0$.
\end{corollary}

\begin{proof}
It follows from the assumptions that
$$\sum_{k=1}^m \sum_{n=1}^j(mj)^{-\alpha-f(\frac{mj}{kn})}\leq (mj)^{1-\alpha-c_0}.$$
Hence,
$$\sum_{m=1}^\infty \sum_{j=1}^\infty(mj)^{-\alpha}P(Y_{m,j}^+>0)\leq \left(\sum_{m=1}^\infty m^{1-\alpha-c_0}\right)^2<+\infty$$ as the right hand side converges for $\alpha > 2 - c_0,$ which completes the proof.
\end{proof}

\begin{remark}  If the conditions of Theorem {\rm \ref{P3}} or Corollary {\rm \ref{cor3}} are satisfied, then for every $\varepsilon > 0$
$$\sum_{m=1}^{+ \infty}\sum_{j=1}^{+ \infty}( mj)^{- \alpha} P(Y_{m,j}^+ > \varepsilon) <\infty$$  as the inequality $P( Y_{m,j}^{+} > \varepsilon) \leq P( Y_{m,j}^+ > 0 )$ holds true.
\end{remark}

Now we proceed with extending Proposition \ref{D}, showing that the rate of convergence is sharp.

\begin{theorem}
\label{P4}
Let  $\{X_{k,n},  k,n \geq 1 \}$  be a double array of independent $\varphi$-subgaussian random variables with $\tau_{\varphi}(X_{k,n})\equiv 1$ satisfying the following assumptions:

\begin{enumerate}[label=(\roman*)]
\item there exists a strictly increasing function $\kappa:\mathbb{R}^+\to\mathbb{R}^+$ such that for every $k,n\geq 1$ and some positive constant C it holds
$$P(X_{k,n}>x) \geq C\exp(-\kappa(x)), \ x>0;$$
\item there exists $x_0>0$ such that
$$\exp(-\kappa(x))\geq C_1\exp(-B\psi(x)),$$ for all $x\geq x_0,$ where $B,C_1>0$;
\item for some $\varepsilon > 0$ $$\sup_{x>x_0}\frac{q_{\varphi}(x+\varepsilon)}{\psi(x)}\leq C_2<+\infty.$$
\end{enumerate}

Then, for any $\alpha < 2-B(1+C_2\varepsilon)$ it holds
\[\sum_{m=1}^{\infty} \sum_{j=1}^{\infty} (mj)^{- \alpha} P(Y_{m,j}^+ >\varepsilon) =+\infty. \]
\end{theorem}

\begin{proof} By the theorem's assumption one can take $g(\cdot)\equiv 1$ and obtain
$$ P(Y_{m,j}^+ >\varepsilon) = P\left(\max _{1 \leq k \leq m, 1 \leq n \leq j} X_{k,n} > \psi^{-1}\left(\ln(mj)\right)+\varepsilon\right) $$
$$ = 1 - \prod_{k=1}^m \prod_{n=1}^j \left(1-P\left(X_ {k,n} \geq \psi^{-1}(\ln (mj))+\varepsilon \right)\right) $$
$$ \geq 1 -\left( 1- C\exp\left(-\kappa\left(\psi^{-1}(\ln (mj))+\varepsilon\right)\right)\right)^{mj}.$$

Using the inequality $1-t\leq e^{-t}, \ t\geq 0,$ one obtains
$$P(Y_{m.j}^{+}\geq \varepsilon)\geq 1 - \exp\left( -Cmj\exp\left(-\kappa\left(\psi^{-1}\left(\ln(mj)\right) + \varepsilon  \right) \right)\right).$$
Then, by the inequality $1-\exp(-t)\geq t\exp(-t), \ t\geq0,$ it follows that
$$P(Y_{m,n}^{+}\geq \varepsilon) \geq Cmj\exp\left(-\kappa\left(\psi^{-1}\left(\ln(mj)\right)+\varepsilon \right) \right) $$
$$\times \exp\left( -Cmj\exp\left(-\kappa\left(\psi^{-1}\left(\ln(mj) \right)+\varepsilon\right) \right)\right).$$

By Lemma \ref{l1} and assumption $(i)$
$$C\exp(-\kappa(x)) \leq \exp(-\psi(x)), \ x\geq 0.$$

Noting that $\kappa(\cdot)$ is an increasing function, one obtains
$$P(Y_{m,j}^+>\varepsilon)\geq Cmj\exp\left(-\kappa\left(\psi^{-1}\left(\ln(mj)\right)+\varepsilon \right) \right)$$
$$\times\exp\left( -Cmj\exp\left(-\kappa\left(\psi^{-1}\left(\ln(mj) \right)\right) \right)\right)$$
$$\geq Cmj\exp\left(-\kappa\left(\psi^{-1}\left(\ln(mj)\right)+\varepsilon \right) \right)\exp\left( -Cmj\exp\left(-\psi\left(\psi^{-1}\left(\ln(mj) \right)\right) \right)\right)$$
$$\geq Cmj\exp\left(-1\right)\exp\left(-\kappa\left( \psi^{-1}\left(\ln(mj)+\varepsilon\right) \right)\right).$$
Then, by assumption $(ii)$
$$P(Y_{m,j}^+>\varepsilon) \geq CC_1\exp(-1)mj\exp\left(-B\psi\left( \psi^{-1}(\ln(mj))\right)+\varepsilon\right).$$

It follows from assumption $(iii)$ that
$$\frac{\psi\left( \psi^{-1}\left(\ln(mj)\right) + \varepsilon\right)}{\psi\left(\psi^{-1}\left(\ln(mj)\right)\right)}=\frac{\int_0^{\psi^{-1}(\ln(mj))}q_{\varphi}(x)dx + \int_{\psi^{-1}(\ln(mj))}^{\psi^{-1}(\ln(mj))+\varepsilon}q_{\varphi}(x)dx}{\int_{0}^{\psi^{-1}(\ln(mj))}q_{\varphi}(x)dx}$$

$$\leq 1 + \frac{\varepsilon q_\varphi(\psi^{-1}(\ln(mj))+\varepsilon)}{\psi(\psi^{-1}\ln(mj))}\leq 1 +C_2\varepsilon,$$ as $q_\varphi(\cdot)$ is an increasing function.

Hence, it holds
$$P(Y_{m,j}^+>\varepsilon) \geq Cmj\exp\left(-B(1+C_2\varepsilon)\psi(\psi^{-1}(\ln(mj))) \right) = C(mj)^{1-B(1+C_2\varepsilon)}.$$

Therefore, \[\sum_{m=1}^{\infty}\sum_{j=1}^{\infty} (mj)^{-\alpha}P(Y_{m,j}>\varepsilon)\geq C\left( \sum_{m=1}^{\infty} m^{1-\alpha-B(1+C_2\varepsilon)} \right)^2 = +\infty,\] when $1-\alpha-B(1+C_2\varepsilon) > -1,$ which completes the proof.   \end{proof}

\section{Theoretical examples}~\label{sec5}
This section provides theoretical examples for important particular classes of $\varphi$-subgaussian distributions. Specifications of functions $\kappa(\cdot)$ and $\varphi(\cdot),$ such that  the obtained theoretical results hold true, are given.

\begin{example} 
{\rm Let $\{X_{k,n}, \ k,n\geq 1\}$ be a double array of standard Gaussian random variables. It is well-known that $Ee^{tX_{k,n}} =  e^{t^2/2}$ which implies that $\{X_{k,n}, \ k,n\geq 1\}$ is the double array of $\varphi$-subgaussian random variables with $\varphi(x) = x^2/2.$  The $\varphi$-subgaussian norm of a Gaussian variable equals to its standard deviation that is~1 in this example, i.e. $ \tau_{\varphi}(X_{k,n}) \equiv 1.$ The Young-Fenchel transform of $\varphi(\cdot)$ is $\psi(x) = x^2/2$ with the density $q_\varphi(x) = x.$

One can easily see that the condition \eqref{cor21} of Corollary \ref{rc} is satisfied. Indeed, for any positive $\varepsilon$ the following integral is finite
$$\int_0^{+\infty}\psi(x)q_{\varphi}(x)\exp\left(-\varepsilon q_{\varphi}(x)\right)dx = \int_0^{+\infty}\frac{x^3}{2}\exp\left(-\varepsilon x\right)dx < +\infty.$$

Let us show that the conditions of Corollary \ref{cor2} are satisfied too.

By \cite{Bi}, for all $x>0$ it holds

$$P(X_{k,n}\geq x) \geq \frac{1}{2\sqrt{2\pi}}(\sqrt{4+x^2} - x)e^{-\frac{x^2}{2}}$$
$$=\sqrt{\frac{2}{\pi}}\frac{e^{-\frac{x^2}{2}}}{\sqrt{4+x^2}+x}= \sqrt{\frac{2}{\pi}}e^{-\frac{x^2}{2}-\ln({\sqrt{4+x^2}+x})}.$$
By Remark \ref{rm6}  it means that $\kappa(x) = \frac{x^2}{2}+\ln({\sqrt{4+x^2}+x})$ and $C = \sqrt{\frac{2}{\pi}}.$ The function $\kappa(x)$ is increasing, positive and $\kappa(x)\geq\ln(2)$ for $x>0.$

As \[r(x) = \kappa'(x) = x + \frac{1+\frac{x}{\sqrt{4+x^2}}}{x+\sqrt{4+x^2}}=x+\frac{1}{\sqrt{4+x^2}}>0, \ x>0,\] it is also  positive.

It follows from
$$r'(x) = 1 - \frac{x}{(4+x^2)^{3/2}}>0, \ x>0,$$  that $r(x), \ x>0,$ is non-decreasing.

Also, it easy to see that $C_0(x) = -\ln(x+\sqrt{4+x^2}),  x > 0.$

Let us show that the assumption \eqref{cor21} is satisfied with these specifications of functions $\kappa(\cdot)$ and $\psi(\cdot)$. Indeed, by the change of variables $x = \psi^{-1}(\ln(y))-\varepsilon$ one obtains  $y = e^{\psi(x+\varepsilon)}$and
$$\int_A^{+\infty}\exp\left(-\frac{Cy}{2}\exp\left( -\kappa(\psi^{-1}(\ln(y))-\varepsilon)\right) \right)dy$$

\begin{equation}
\label{ex2eq}
=\int_{A^{'}}^{+\infty}q_{\varphi}(x+\varepsilon)\exp\left( \psi(x+\varepsilon) -\frac{C}{2}e^{\psi(x+\varepsilon) - \kappa(x)}\right)dx,
\end{equation}
where $A^{'} = \psi^{-1}(\ln(A))-\varepsilon.$

By Bernoulli's inequality
$$\psi(x+\varepsilon) - \kappa(x) = \frac{x^2}{2}\left(\left(1+\frac{\varepsilon}{x}\right)^2 - 1\right) - \ln(\sqrt{4+x^2}+x) \geq \varepsilon x - \ln(\sqrt{4+x^2}+x).$$

As a polynomial growth is faster than the logarithmic one, the integral in \eqref{ex2eq} is bounded from above by

$$\widetilde{C}\int_{A^{'}}^{+\infty}(x+\varepsilon)\exp\left(\frac{(x+\varepsilon)^2}{2} - e^{\widetilde{\varepsilon} x} \right)dx$$ for some $\tilde{C},\widetilde{\varepsilon}>0.$

As exponentials grow faster than polynomials, for sufficiently large $x$

$$\frac{(x+\varepsilon)^2}{2} - \exp\left(\widetilde{\varepsilon}x\right) \leq -{C'}\exp\left( \widetilde{\varepsilon} x \right)$$ and
$$\exp\left(-{C'} \exp\left( \widetilde{\varepsilon}x\right)\right) \leq \exp\left( -C''x \right),$$ for some positive constants ${C'}$ and $C''.$

The assumption \eqref{cor22} is satisfied too. Indeed, the assumption \eqref{cor22} takes the form

$$\int_{A}^{\infty}\psi(y)q_{\varphi}(y)\exp\left(\psi(y)-\frac{C}{2}\exp\left( C_0(y) +\varepsilon r(y-\varepsilon) \right)\right)dy$$
$$=\int_A^{+\infty}\frac{y^3}{2}\exp\left(\frac{y^2}{2} -\frac{C}{2(y+\sqrt{4+y^2})}\exp\left(\varepsilon\left(y-\varepsilon + \frac{1}{\sqrt{4+(y-\varepsilon)^2}}\right)\right)  \right).$$
The last integral is finite because
$$\frac{y^2}{2} -\frac{C}{2(y+\sqrt{4+y^2})} \exp\left(\varepsilon\left(y-\varepsilon + \frac{1}{\sqrt{4+(y-\varepsilon)^2}}\right)\right)<-\exp\left(\frac{\varepsilon y}{2}\right)$$ for sufficiently large $y.$

By Theorem \ref{th1} one gets that $\lim_{m\vee j\to\infty} Y_{m,n} =0$ a.s.}
\end{example}

\begin{example}
\label{ex}
{\rm  Let $\{X_{k,n}, \ k,n\geq 1\}$ be a double array of independent identically distributed reflected Weibull random variables with the probability density
\begin{equation*}
    p(x) =
        \frac{\theta}{2b}\left(\frac{|x|}{b}\right)^{\theta-1}e^{-(|x|/b)^\theta},\quad \theta > 0,\ b>0.
\end{equation*}
Consider reflected Weibull random variables with $\theta >1$ and $b>0.$ They belong to the $\varphi$-subgaussian class. Indeed, tails of reflected Weibull random variables equal
\[
P(|X_{k,n}| > x) = e^{-(\frac{x}{b})^\theta}, \ x\geq 0.
\]
Hence, by \cite[Corollary 4.1, p.~68]{BK} $\{X_{k,n}, \ k,n\geq 1\}$ is the double-array of $\varphi$-subgaussian random variables,  where
\[\psi(x) = \left( \frac{x}{b}\right)^\theta,\quad \varphi(x) = \frac{\theta-1}{\theta} \left( \frac{\theta}{b^\theta}\right)^{1/(\theta-1)}x^{\theta/(\theta-1)},\quad x\geq 0,\]
see \cite[Example 2.5, p.~46]{BK}, and $\tau_{\varphi}(X_{k,n}) \equiv c < +\infty.$ The density of $\psi(x)$ is $q_\varphi(x) = {\theta} x^{\theta-1}/{b^\theta}, \ x\geq 0.$

Let us chose a such value of the parameter $b$ that $\{X_{k,n}, \ k,n\geq 1\}$ is the double-array of $\varphi$-subgaussian random variables with  $\varphi$-subgaussian norms $\tau_{\varphi}(X_{k,n}) \equiv c \leq 1,$ see Section \ref{sec6}. We will show that in this case the conditions of Corollaries~{\rm 1} and {\rm 2} are satisfied.

The conditions of Corollary \ref{rc} are satisfied because $\tau_{\varphi}(X_{k,n}) \equiv c$ and the following integral is finite for all positive $\varepsilon$
$$\int_0^{+\infty}\psi(x)q_{\varphi}(x)\exp\left(-\varepsilon q_{\varphi}(x)\right)dx=\frac{\theta}{b^{2\theta}}\int_0^{+\infty}x^{2\theta-1}\exp\left(- \frac{\varepsilon\theta}{b^{\theta}}x^{\theta-1}\right)dx<+\infty.$$

Let us show that the conditions of Corollary \ref{cor2} are satisfied too. By Remark~\ref{rm6} and the equality $P(X_{k,n} > x) = \frac{1}{2}e^{-(\frac{x}{b})^\theta}, \ x\geq 0,$ it follows that $\kappa(x) = \psi(x) = \left(\frac{x}{b} \right)^\theta,$ $C = {1}/{2},$ and $r(x)  = q_\varphi(x).$ Hence, \[ \psi(x) - \kappa(cx) =(1-c^\theta)\left(\frac{x}{b} \right)^\theta\geq C_0(x) =0\] because $g(\cdot)\equiv c\leq 1.$

The assumption \eqref{cor21} can be rewritten as
$$\int_A^{+\infty}\exp\left(-\frac{y}{4}\exp\left(-\kappa\left( \psi^{-1}(\ln(y))-\varepsilon\right)\right)\right)dy$$
$$ = \int_A^{+\infty}\exp\left(-\frac{y}{4}\exp\left(-\psi\left( \psi^{-1}(\ln(y))-\varepsilon\right)\right)\right)dy.$$
Let use the change of variables $x = \psi^{-1}(\ln(y))-\varepsilon.$ Then, $y = e^{\psi(x+\varepsilon)}$ and the above integral equals to
$$\int_{A'}^{+\infty}q_{\varphi}(x+\varepsilon)\exp\left(\psi(x+\varepsilon)-\frac{e^{\psi(x+\varepsilon)-\psi(x)}}{4}\right)dx$$
\begin{equation}\label{exeq2}
 = \frac{\theta}{b^{\theta}} \int_{A'}^{+\infty}(x+\varepsilon)^{\theta-1}\exp\left(\frac{(x+\varepsilon)^{\theta}}{b^{\theta}} -\frac{1}{4}\exp\left(\frac{(x+\varepsilon)^\theta-x^\theta}{b^\theta} \right)\right)dx,
\end{equation} where $A' = \psi^{-1}(\ln(A))-\varepsilon.$

By Bernoulli's inequality
$$(x+\varepsilon)^\theta - x^\theta = x^\theta\left(\left(1+\frac{\varepsilon}{x}\right)^\theta-1\right)\geq \varepsilon\theta x^{\theta-1}$$ and the integral in \eqref{exeq2} is bounded by
$$\frac{\theta}{b^\theta}\int_{A'}^{+\infty}(x+\varepsilon)^{\theta-1}
\exp\left(\frac{(x+\varepsilon)^\theta}{b^{\theta}}-\frac{1}{4}\exp\left(\frac{\varepsilon \theta}{b^\theta} x^{\theta - 1} \right)\right)dx.$$

As exponentials grow faster than polynomials we obtain that for sufficiently large~$x$
$$\frac{(x+\varepsilon)^\theta}{b^\theta}-\frac{1}{4}\exp\left(\frac{\varepsilon \theta}{b^\theta} x^{\theta - 1} \right)\leq -C\exp\left( \frac{\varepsilon\theta}{b^\theta}x^{\theta-1}   \right)$$ and
$$\exp\left(-C\exp\left(\frac{\varepsilon\theta}{b^\theta}x^{\theta-1}  \right)\right) \leq \exp\left( - \widetilde{C} x^{\theta -1}\right)$$ for some positive constants $C$ and $\widetilde{C}.$

Finally, as for $\theta > 1$
$$\frac{\theta}{b^\theta}\int_{A'}^{\infty}(x+\varepsilon)^{\theta-1}\exp\left( -\widetilde{C}x^{\theta-1}\right)dx<+\infty$$ we obtain \eqref{cor21}.

Now, let us check the assumption \eqref{cor22}. In our case, it takes the form
$$\int_A^{+\infty}\psi(y)q_{\varphi}(y)\exp\left(\psi(y) -\frac{\exp\big( \varepsilon q_{\varphi}(y-\varepsilon) \big)}{4} \right)dy$$
$$=\frac{\theta}{b^\theta}\int_A^{+\infty}y^{2\theta-1}\exp\left(\left(\frac{y}{b}\right)^\theta - \frac{\exp\left(\frac{\varepsilon\theta}{b^\theta}(y-\varepsilon)^{\theta-1}\right)}{4}\right)dy.$$
Again, as $\theta > 1$ then $(\frac{y}{b})^\theta - \exp(\frac{\varepsilon\theta}{b^\theta}(y-\varepsilon)^{\theta-1}) <-y$ for sufficiently large $y,$ which means that the integral is finite. Thus, Corollaries {\rm 1} and {\rm 2} hold true. Moreover, by Theorem \ref{th1} one gets that $\lim_{m\vee j\to\infty} Y_{m,n} =0$ a.s.}

\end{example}

\section{Numerical examples}~\label{sec6}
This section provides numerical examples that confirm the obtained theoretical results. By simulating double arrays of random variables satisfying Theorem \ref{th1}, we show that the running maxima functionals of these double arrays converge to $0$, as the size of observation windows tends to infinity. As the rate of convergence is very slow to better illustrate asymptotic behaviour we selected arrays with constant  $\varphi$-subgaussian norms close to one.

Consider a double array $\{X_{k,n}, k,n \ge 1\}$ that consists of independent reflected Weibull random variables (see Example \ref{ex}) with the parameters $\theta = 9$ and $b = 1.25.$ These values of $\theta$ and $b$ were selected to get $\tau_\varphi(X_{k,n})\le 1$ as in  Example~\ref{ex}.  The probability density function of the underlying random variables $X_{k,n}, k,n \ge 1,$ and a realization of the double array  $\{X_{k,n}, k,n \ge 1\}$  in a square window are shown in Figure \ref{fig1}.

\begin{figure}[htb!]
\begin{subfigure}{0.5\textwidth}
  \centering
  \includegraphics[width=1\linewidth,height=6cm,trim=0 5mm 0 1cm,clip]{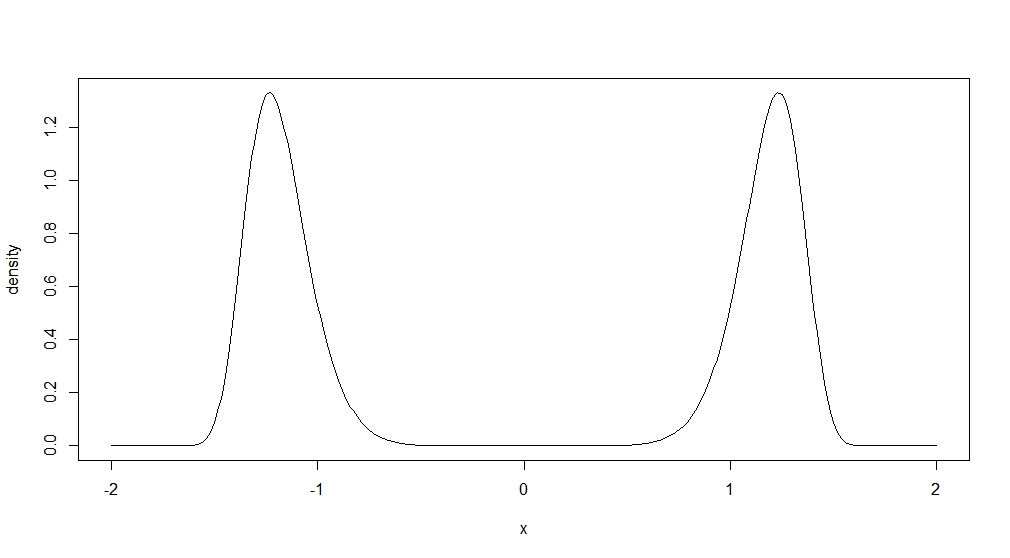}
  \caption {Probability density function}
  \label{fig1a}
\end{subfigure}%
\hspace{0mm}
\begin{subfigure}{0.6\textwidth}
  \includegraphics[width=1\linewidth,height=6cm,trim=4.5cm 0mm 1.2cm 0mm,clip]{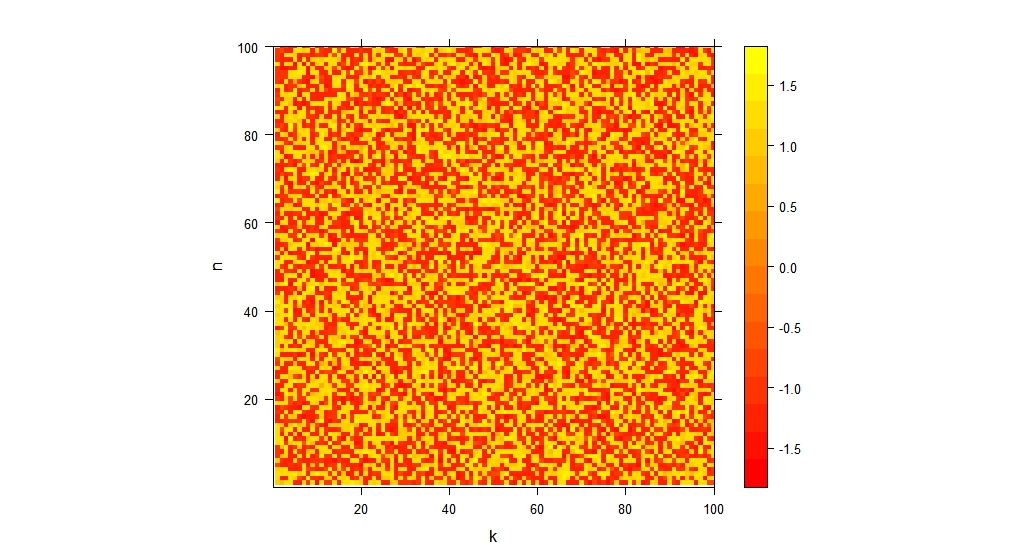}
  \caption{Realization of the double array}
  \label{fig1b}
\end{subfigure}
\caption{Double array of Weibull random variables}
\label{fig1}
\end{figure}

The underlying random variables $X_{k,n}, \ k,n\geq 1,$ are $\varphi$-subgaussian random variables with
\[\psi(x) = \left(\frac{x}{1.25}\right)^9,\quad \varphi(x) = \frac{8}{9}\left( \frac{9}{1.25^9}\right)^{1/8} x^{9/8},\quad x\geq 0.\]

A calculation of the $\varphi$-subgaussian norm by using Definition~\ref{def4} is not trivial in a general case and may require different approaches. The following method was used to estimate the $\varphi$-subgaussian  norm of $X_{k,n}, \ k,n \ge 1.$ By \cite[Lemma 4.2, p.~65]{BK} the $\varphi$-subgaussian norm allows the representation
\[
\tau_\varphi(X_{k,n}) = \sup_{\lambda \neq 0}\frac{\varphi^{(-1)}\left(\ln\left( E\exp(\lambda X_{k,n})\right)\right)}{|\lambda|}.
\]
For the reflected Weibull random variables the above expectation can be calculated~as
\[ E\exp(\lambda X_{k,n}) = \int_{\mathbb{R}}e^{\lambda x}p(x)dx = \int_{0}^{+\infty}e^{\lambda x}p(x)dx + \int_{0}^{+\infty}e^{-\lambda x}p(x)dx\]
\[ = \frac{1}{2}\left(MGF(\lambda)+(MGF(-\lambda)\right),\]
where $MGF(\cdot)$ denotes the moment generating function of the corresponding Weibull distribution and is given by
$$MGF(\lambda) =  \sum_{n=0}^{+\infty} \frac{\lambda^n b^n}{n!} \Gamma\left(1+\frac{n}{\theta} \right).$$
By using this representation one gets
$$E\exp(\lambda X_{k,n})=\frac{1}{2}\left(\sum_{n = 0}^{+\infty}\frac{\lambda^n b^n}{n!}\Gamma\left(1+\frac{n}{\theta}\right) + \sum_{n = 0}^{+\infty}\frac{(-\lambda)^n b^n}{n!}\Gamma\left(1+\frac{n}{\theta}\right) \right) $$
$$=\sum_{n = 0}^{+\infty}\frac{\lambda^{2n} b^{2n}}{(2n)!}\Gamma\left(1+\frac{2n}{\theta}\right).$$
Thus, for sufficiently large $M$ the $\varphi$-subgaussian norm of the reflected Weibull random variables can be approximated by
$$\tau_\varphi(X_{k,n}) \approx  \sup_{\lambda \neq 0}\frac{\varphi^{(-1)}\left(\ln\left(\sum_{n = 0}^{M}\frac{\lambda^{2n} b^{2n}}{(2n)!}\Gamma\left(1+\frac{2n}{\theta}\right)\right)\right)}{|\lambda|}$$
$$=\sup_{\lambda \neq 0} \left(\frac{9}{8}\left(\frac{1.25^9}{9}\right)^{1/8}\frac{\left(\ln\left(\sum_{n = 0}^{M}\frac{\lambda^{2n} b^{2n}}{(2n)!}\Gamma\left(1+\frac{2n}{\theta}\right)\right)\right)^{8/9}}{|\lambda|}\right).$$

As $(2n)!$ increases very quickly even small values of $M$ provide a very accurate approximation of the series and the norm.

Figure \ref{fig_norm} shows the graph of the function under the supremum and the supremum value for $M = 50.$ As the function is symmetric only the range $\lambda>0$ is plotted. For $\theta = 9$ and $b = 1.25$ the supremum is attained at $\lambda=8.5801$ and the estimated value of the  norm $\tau_\varphi(X_{k,n})$ is 0.997.   Thus, the double array $\{X_{k,n}, k,n \ge 1\}$ satisfies the conditions of Theorem \ref{th1}, see Example~\ref{ex}.

\begin{figure}[htb!]
  \centering
  \includegraphics[width=1\linewidth,height=6cm, width = 8cm, trim=0 4mm 0 20mm,clip]{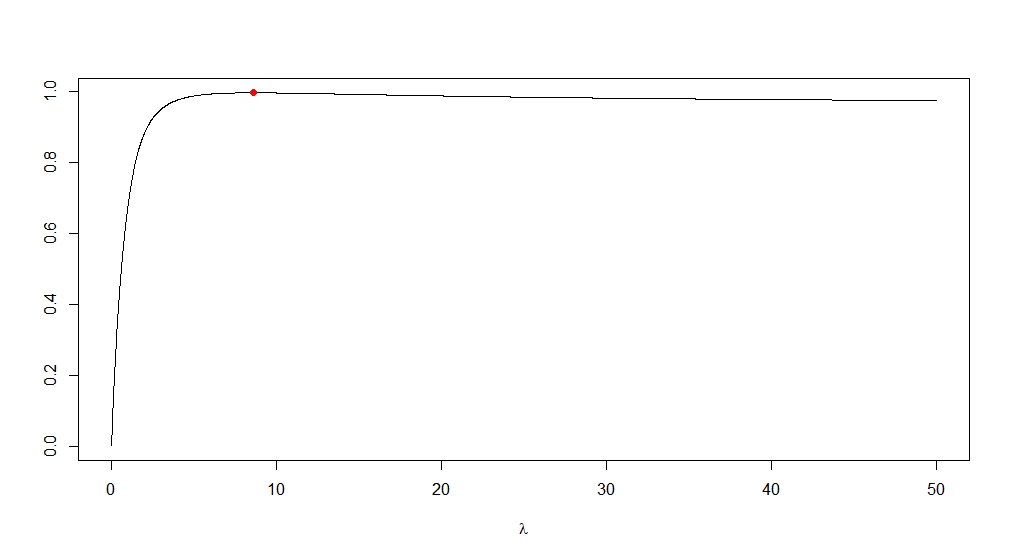}
  \caption {Estimation of $\tau_\varphi(X_{k,n}).$}
  \label{fig_norm}
\end{figure}

\begin{figure}[htb!]
\begin{subfigure}{0.5\textwidth}
  \centering
  \includegraphics[width=1\linewidth,height=7cm,trim=0 3mm 0 0,clip]{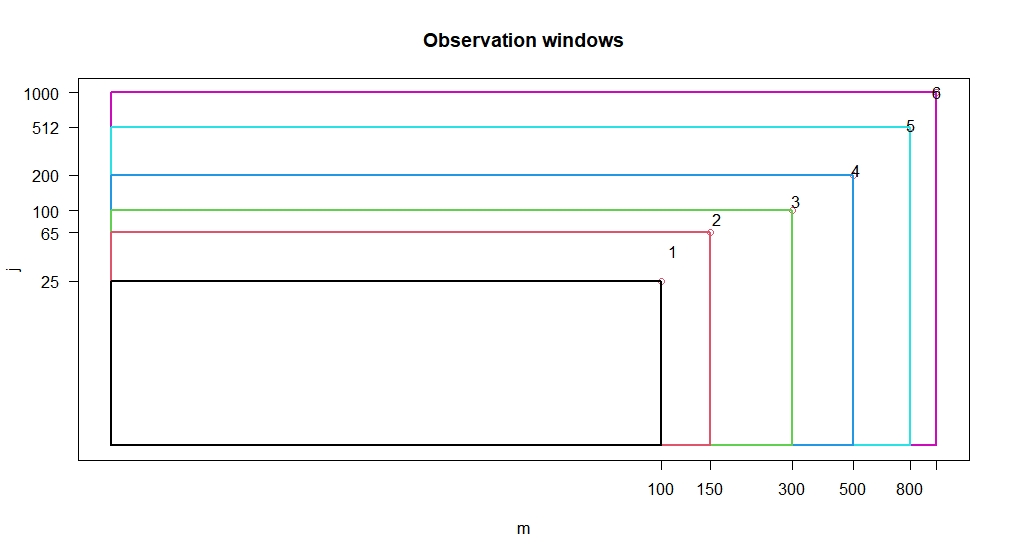}
  \caption {The first set}
  \label{fig2a}
\end{subfigure}%
\begin{subfigure}{0.5\textwidth}
  \centering
  \includegraphics[width=1\linewidth,height=7cm,trim=0 4mm 0 0mm,clip]{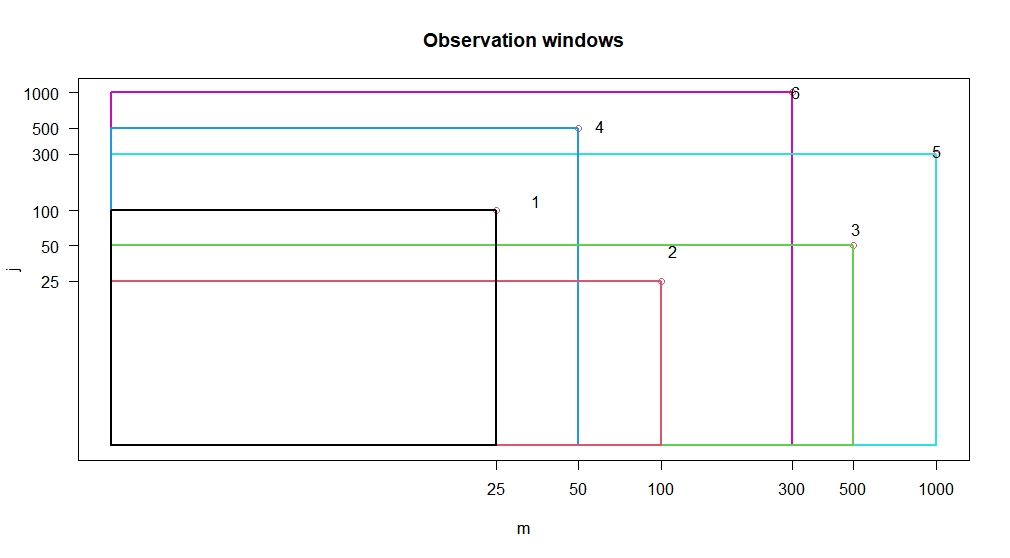}
  \caption{The second set}
  \label{fig2b}
\end{subfigure}
\caption{Observation windows.}
\label{fig2}
\end{figure}

Then, 1000 realizations of the double array $\{X_{k,n}, k,n \ge 1\}$ in the square region $\square(1000) = \{(m, j) : m,j \leq 1000, m,j \in\mathbb{N}\}$ were generated. Using the obtained realizations of the double array, values of $Y_{m,j}$ were computed for two sets of observation windows. The windows are shown in Figure~\ref{fig2} by using logarithmic scales for $x$ and $y$ coordinates. For the set of observation windows in Fig.~\ref{fig2a} and a realization of the reflected Weibull random array, the corresponding running maximas are shown in Fig \ref{fig3a}. For all rectangular observation windows inside $\square(1000)$ locations of maximas are shown in Fig~\ref{fig3b}. The locations are very sparse and the majority of them is concentrated closely to the left and bottom borders of $\square(1000).$

\begin{figure}[htb!]
\vspace{-0.5cm}
\begin{subfigure}{0.48\textwidth}
  \centering
  \includegraphics[width=1\linewidth,height=6.7cm,trim=0 0mm 0 0,clip]{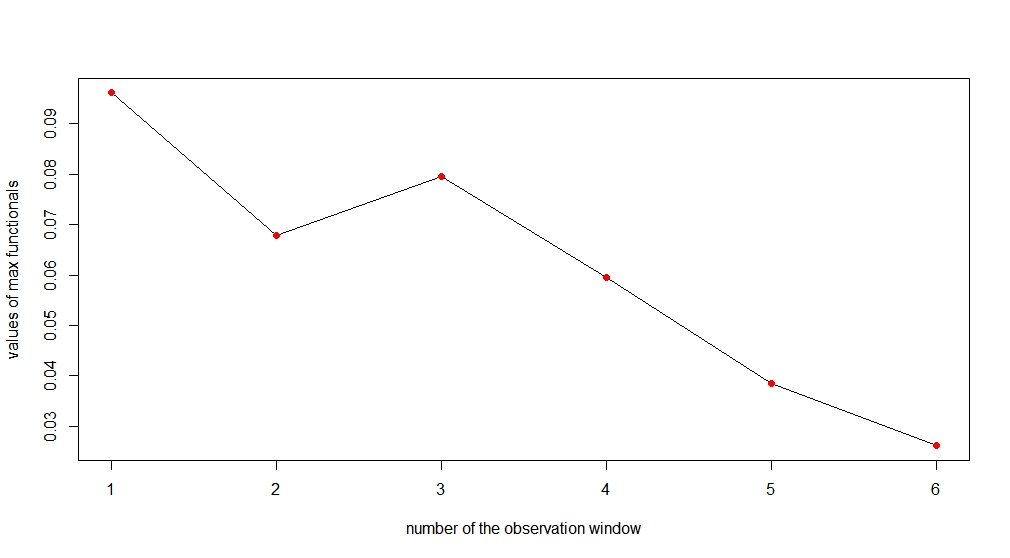}
  \caption {Running maxima $Y_{m,j}$ for the first set of windows}
  \label{fig3a}
\end{subfigure}%
\hspace{1mm}
\begin{subfigure}{0.48\textwidth}
  \centering
  \includegraphics[width=1\linewidth,height=6.7cm,trim=0 4mm 0 4mm,clip]{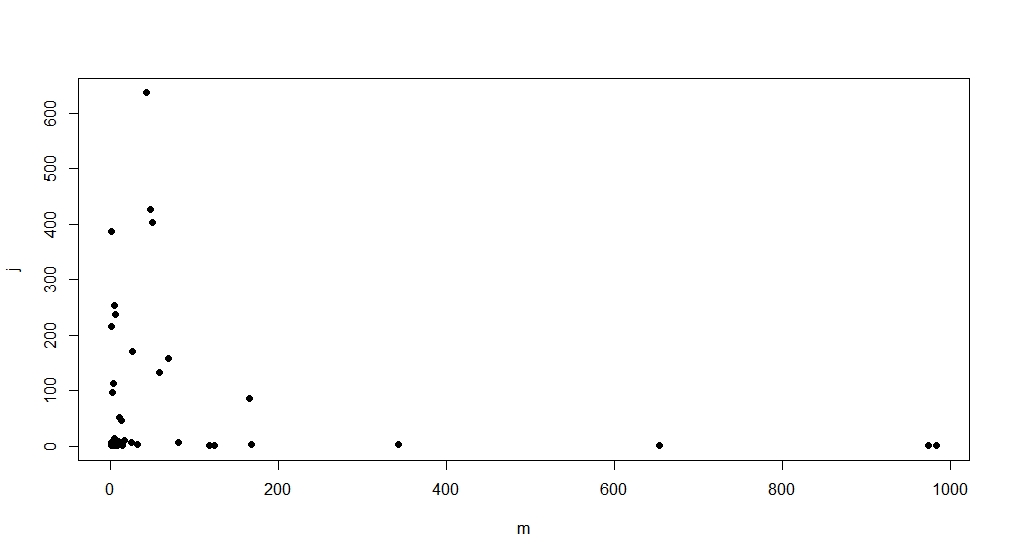}
  \caption{Locations of maximas for all rectangular subwindows in $\square(1000)$}
  \label{fig3b}
\end{subfigure}
\caption{Running maxima of a realization over a set of windows}
\label{fig3}
\end{figure}

For the simulated 1000 realizations and the corresponding sets of the observation windows from Figure~\ref{fig2} the box plots of running maxima functionals $Y_{m,j}$   are shown in Figure~\ref{fig4}. It is clear that the distribution of the running maxima concentrates around zero when the size of the observation window increases, but the rate of convergence seems to be rather slow.
\begin{figure}[htb!]
\begin{subfigure}{0.45\textwidth}
  \centering
  \includegraphics[width=1\linewidth,height=6cm,trim=0 4mm 0 20mm,clip]{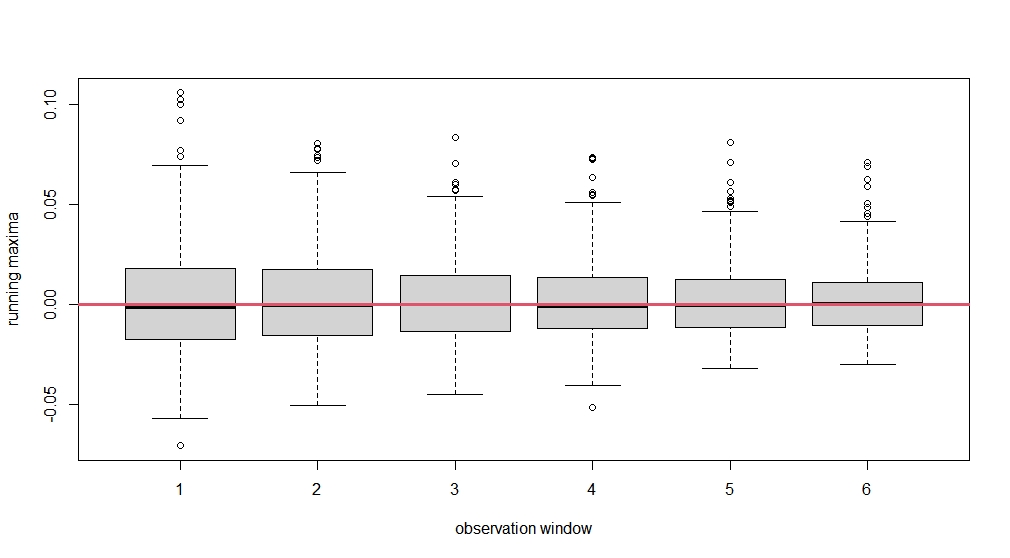}
  \caption {The first case}
  \label{fig4a}
\end{subfigure}%
\hspace{12mm}
\begin{subfigure}{0.45\textwidth}
  \centering
  \includegraphics[width=1\linewidth,height=6cm,trim=0 4mm 0 20mm,clip]{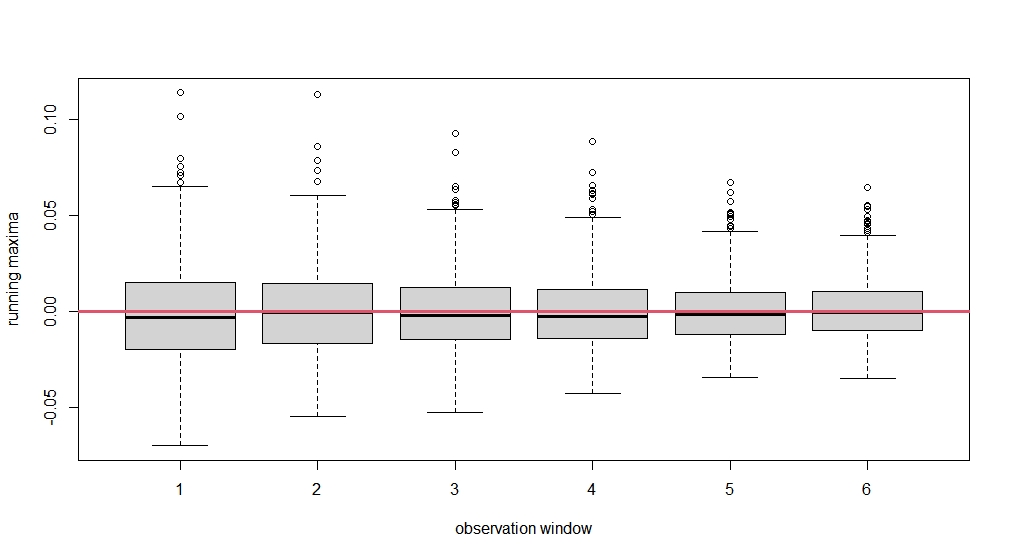}
  \caption{The second case}
  \label{fig4b}
\end{subfigure}
\caption{Box plots of running maximas for two sets of observation windows.}
\label{fig4}
\end{figure}

Table~\ref{rmse1} shows the corresponding Root Mean Square Error (RMSE) of  the running maxima functionals $Y_{m,j}$ from Figure~\ref{fig4}, the table confirms the convergence of $Y_{m,j}$  to zero when the observation window increases.

\begin{table}[hb]
\begin{center}
 \begin{tabular}{|c|c c c c c c|}
 \hline
$\rm{Observation \ window}$ \rule[-2mm]{0pt}{7mm}& 1 & 2 & 3 & 4 & 5 & 6  \\
 \hline
The first case \rule[-2mm]{0pt}{7mm} & 0.026 & 0.022 & 0.021 & 0.019 &  0.017 & 0.016  \\ [1ex]
 \hline
The second case \rule[-2mm]{0pt}{7mm} & 0.038 & 0.032 & 0.024 & 0.024 & 0.018 & 0.017  \\ [1ex]
\hline
\end{tabular}
 \caption{RMSE of  $Y_{m,j}$.\label{rmse1}}
 \end{center}\vspace{-0.5cm}
\end{table}

\begin{figure}[htb!]
  \centering
  \includegraphics[width=1\linewidth,height=6.5cm, width = 8cm, trim=0 8mm 0 20mm,clip]{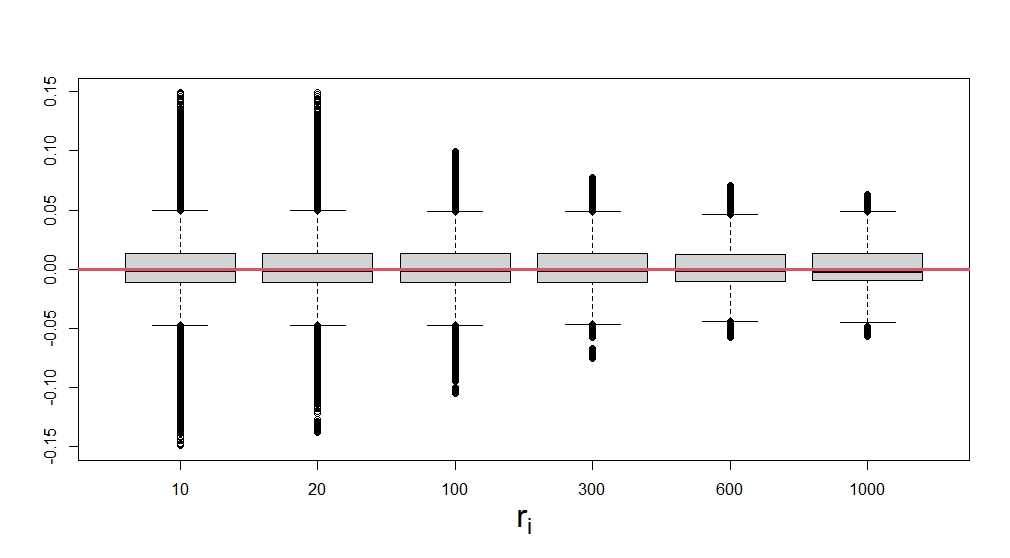}
  \caption {Boxplots of running maximas $Y_{m,j}$ for 6 groups.}
  \label{fig5}
\end{figure}

Finally, to demonstrate the $\lim(\max)$ convergence 1000 simulated realizations of the reflected Weibull double array were used. The running maxima functionals $Y_{m,j}$ were calculated for all possible pairs $(m,j),$  $m,j = 1,2,...,1200.$ In Figure~\ref{fig5} the boxplots of the obtained values of $Y_{m,j}$ were computed for 6 groups depending on values of the parameter $r = m \vee j$ in the corresponding observation subwindows. The lower bound for the parameter $r$ increases with the group  number. Namely, in group $i=1,...,6$ the values $r\ge r_i,$ where $r_i=10,20,100,300,600, 1000.$ The obtained boxplots  in Figure~\ref{fig5} confirm the $\lim(\max)$ convergence.

\section{ Conclusions and the future studies}~\label{sec7}
The asymptotic behaviour of running maxima of random double arrays was investigated. The conditions of the obtained results allow to consider a wide class of $\varphi$-subgaussian random fields and are weaker than even in the known results for the one-dimensional case. The rate of convergence was also studied. The results were derived for a general class of rectangular  observation windows and $\lim(\max)$ convergence.

In the future studies, it would be interesting to extend the obtained results to:

- the case of $n$-dimension arrays,

- other types of observation windows,

- continuous $\varphi$-subgaussian random fields,

- different types of dependencies.

 \section*{Acknowledgements}
 This research was supported by La Trobe University SEMS CaRE Grant "Asymptotic analysis for point and interval estimation in some statistical models". 
 
 This research includes computations using the Linux computational cluster Gadi of the National Computational Infrastructure (NCI), which is supported by the Australian Government and La Trobe University.

\end{document}